\RequirePackage{fix-cm}
\documentclass[smallextended]{svjour3}       
\smartqed  
\usepackage[linesnumbered,boxed, longend]{algorithm2e}   
\usepackage{algpseudocode}
\usepackage{amsfonts}
\usepackage{graphicx}
\usepackage{tabularx,multirow}
\usepackage{caption}
\usepackage{amssymb}
\usepackage{etex}
\usepackage{array}
\usepackage{booktabs} 
\usepackage{amsmath}

\usepackage[square,comma,numbers,sort&compress]{natbib}

\renewenvironment{proof}{{\emph{Proof. }}}{}
\usepackage{url}
\usepackage{subfigure}

\newtheorem{myobservation}{Observation}
\usepackage{pstricks}
\usepackage{pst-plot}
\usepackage{pst-eps}
\usepackage{pst-grad}
\usepackage{pgfplots}
\usepackage{setspace}
\usepackage{mathtools, nccmath}
\titlerunning{\scriptsize On the complexity of sequentially lifting cover inequalities for the knapsack polytope}

\usepackage{lineno}
\usepackage{lipsum}
\renewcommand\figurename{\textbf{Figure}}

\def\rev#1{{{#1}}}
\def\hilight#1{{{#1}}}

\def\highlight#1{{{#1}}}

\begin{document}

\title{On the complexity of sequentially lifting cover inequalities for the knapsack polytope \thanks{This work was supported by the Chinese Natural Science Foundation
		(Nos. 11631013, 11331012) and the National 973 Program of China (No. 2015CB856002).}
}

\author{Wei-Kun Chen \and {Yu-Hong Dai}${}^\ast$}

\institute{${}^\ast$ Corresponding author
\\LSEC, ICMSEC,
Academy of Mathematics and Systems Science, Chinese Academy of Sciences, Beijing, China \and School of Mathematical Sciences, University of Chinese Academy of Sciences, Beijing, China. 
\email{cwk@lsec.cc.ac.cn \and dyh@lsec.cc.ac.cn}
}


\date{\today}

\maketitle

\begin{abstract}
	The well-known \highlight{sequentially} {\text{lifted cover inequality}} is widely employed in solving mixed integer programs. However, it is still an open question whether \highlight{a sequentially} lifted cover inequality \rev{can} be computed in polynomial time for a given minimal cover (Gu, Nemhauser, and Savelsbergh, INFORMS J. Comput., 26: 117--123, 1999). We show that this problem is $\mathcal{NP}$-hard, thus giving a negative answer to the question. 
\end{abstract}
\section{Introduction}
\label{sec-introduction}

The {\emph{lifted cover inequality}} (LCI) is \rev{a} well-known cutting plane for mixed integer programs. 
\highlight{Given \rev{the so-called cover inequality}, in order to obtain an LCI, we may use
the {\textit{lifting} technique}. Using different lifting procedures, several types of LCIs 
have been studied in the literatures, see \cite{Balas1975,Wolsey1975,Hammer1975,Balas1978a,Hartvigsen1992,Gu2000,Atamturk2005a,Letchford2019}. In this paper, we are 
concentrated on the {\textit{sequential}} LCI, that is, 
the variables are sequentially lifted one at a time.} 
The sequential LCI was first studied in \highlight{\cite{Padberg1975a,Wolsey1975}}. Its effectiveness on using as \rev{a} cutting plane was demonstrated in \cite{Crowder1983}, see also \cite{VanRoy1987,Hoffman1991a,Gu1998,Wolter2006,Kaparis2010a}.  
To lift each variable, a knapsack problem is required to be solved to compute the lifting coefficient. Under certain conditions, the \highlight{sequential} LCI can be computed in polynomial time, see \cite{Zemel1989,Nemhauser1988,Gu1995}. \rev{In general, however,} the complexity of computing \highlight{a sequential} LCI is still unknown. \rev{This was explicitly mentioned in \cite{Gu1999} as an open question.}
\vspace{0.2cm}

``We show that this dynamic programming algorithm may take exponential time to compute a sequential LCI that is not simple. It is still an open question whether an arbitrary LCI can be computed in polynomial time for a given minimal cover $ C $."
\vspace{0.2cm}

\rev{The above question was} also cited as an open question in \cite{Gu1995,VanHoesel2002,Atamturk2002c,Gupta2005}. We will give a
negative answer to the question by showing that the problem of computing a \highlight{sequential} LCI is $ \mathcal{NP} $-hard. Thus, unless $ \mathcal{P}=\mathcal{NP} $, there exists no polynomial time algorithm to compute \highlight{a sequential} LCI.

\rev{This paper is organized as follows. In Section \ref{sec-clci}, we review how to compute \highlight{a sequential} LCI.
In Section \ref{sec-example}, we \rev{describe} the elegant example by Gu \cite{Gu1995}, which provides exponential lifting coefficients. The main result is given in the last section, which shows the $\mathcal{NP} $-hardness of the problem of computings \highlight{a sequential} LCI. }
\section{Computing \rev{a} \highlight{sequentially} lifted cover inequality}
\label{sec-clci}
Consider the knapsack set $X : = \{ x \in \mathbb{B}^n \, : \, a^Tx  \leq b  \}$,
where $ b\in\mathbb{Z}_+$ and $a=(a_1,\cdots,a_n)^T\in\mathbb{Z}_{+}^{n}$ are given.
A subset $ C \subseteq N := \{ 1,\ldots,n \}$ is called a {\textit{cover}} of $ X $ if
the sum of the items $a_i$'s over $C$ exceeds the {\emph{knapsack capacity}} $b$; i.e.,
$ \sum_{i\in C}a_i > b$. A cover $C$ is a {\textit{minimal cover}} if and only if
\begin{equation*}
	\sum_{i \in C \backslash\{ j \}} a_i \leq b\quad \text{for all } j \in C.
\end{equation*}

For any subsets $ N_0 $ and $ N_1 $ of $ N$ with $ N_0 \cap N_1 = \varnothing $, denote $ X(N_0, N_1) $ be the following restriction set of $X$,
\begin{equation*}
X(N_0, N_1) = X \cap \{ x \in \mathbb{B}^n \,:\, x_i = 0\ \text{for } i \in N_0; \ x_i =1\ \text{for } i \in N_1 \}.
\end{equation*}
It is easy to see that, for each cover $C$, the {\textit{cover inequality}}
\begin{equation}
	\label{cover}
	\sum_{i \in C} x_i \leq c -1
\end{equation}
is valid for $  X(N \backslash C, \varnothing) $, where $ c := |C| $ is the cardinality  of $ C $.
The cover inequality \eqref{cover} is facet defining for $\text{conv}(X(N \backslash C, \varnothing))$, which is the convex hull of $ X(N\backslash C, \varnothing)$, if and only if $C$ is a minimal cover (see for example \hilight{\cite{Conforti2014}}).

We can consider to fix some variables to be ones as well. Assume that $(C, N_0, N_1)$ is a partition of $ N $ and denote
$\bar{b}= b - \sum_{i\in N_1} a_i$. In this case, the inequality \eqref{cover} is facet defining for $ \text{conv}(X(N_0,N_1)) $
if and only if $C$ is a minimal cover of $ X(N_0, N_1) $; i.e.,
\begin{eqnarray*}
	\sum_{i\in C}a_i > \bar{b}; \quad& &   \sum_{i \in C \backslash\{ j \}} a_i \leq \bar{b} \quad \text{for all } j \in C.
\end{eqnarray*}
Throughout this paper, we shall assume that $(C, N_0, N_1)$ be a partition of $ N $ and $C$ is a minimal cover of $ X(N_0, N_1) $.

In general, however, the inequality \eqref{cover} may not be valid for $ X $ if $ N_1 \neq \varnothing $. Furthermore, even if
$ N_1 = \varnothing $, such an inequality may not represent a facet of $ \text{conv}(X) $.
To obtain a strong valid inequality, we can lift the variables in $ N_0  \cup N_1 $ one at a time by solving an optimization problem sequentially.
More precisely, let $ l_1,\ldots,l_{n-c} $ be a lifting sequence such that $ N_0 \cup N_1 = \{ l_1,\ldots,l_{n-c} \}$ and
\begin{equation*}
	\sum_{i\in C} x_i + \sum_{i \in N'_0} \alpha_i x_i + \sum_{i\in N'_1} \beta_i x_i  \leq c -1 + \sum_{i \in N'_1}\beta_i
\end{equation*}
be the inequality obtained so far, where $ N'_0 \subseteq N_0 $, $ N'_1 \subseteq N_1 $ and $ N'_0 \cup N'_1 = \{ l_1,\ldots,l_j \} $. To lift the variable $ x_{l_{j+1}} $, we need to solve a knapsack subproblem depending on whether $l_{j+1}$ belongs to $N_0$ or $N_1$. 
\highlight{We follow \cite{Gu1998} in referring to lifting the variable $ x_{l_{j+1}} $
as {\textit{up-lifting}} if $ l_{j+1} \in N_0 $ and {\textit{down-lifting}} if 
		$ l_{j+1} \in N_1 $.} 

\begin{itemize}
	\item [(i)] {\textit{Up-lifting}}. If $ l_{j+1} \in N_0 $, compute the lifting coefficient $ \alpha_{l_{j+1}} $ by solving the lifting problem
\begin{equation}
\label{uplift}
\begin{aligned}
\alpha_{l_{j+1}} =  & \ {\text{min}}
  & &  c - 1 + \sum_{i \in N'_1} \beta_i -  \sum_{i \in C}x_i - \sum_{i\in N'_0 }\alpha_i x_i  -\sum_{i \in N'_1} \beta_i x_i  \\
  &  \ \text{s.t.}  & &           \sum_{i \in C \cup N'_0 \cup N'_1} a_i x_i \leq \bar{b} + \sum_{i \in N'_1}a_i - a_{l_{j+1}} , \ x\in \mathbb{B}^{c+j}  .
\end{aligned}
\end{equation}
	\item [(ii)]  {\textit{Down-lifting}}. If $ l_{j+1} \in N_1 $, compute the lifting coefficient $ \hilight{\beta_{l_{j+1}}} $ by solving the lifting problem
	\begin{equation}
	\label{downlift}
	\begin{aligned}
	\beta_{l_{j+1}} =    &\  {\text{max}}
	& & \sum_{i \in C}  x_i + \sum_{i\in N'_0 }\alpha_i x_i  +\sum_{i \in N'_1} \beta_i x_i - c + 1 - \sum_{i \in N'_1} \beta_i    \\
	&  \ \text{s.t.}          & & \sum_{i \in C \cup N'_0 \cup N'_1} a_i x_i \leq \bar{b} + \sum_{i \in N'_1}a_i + a_{l_{j+1}}  , \ x\in \mathbb{B}^{c+j}.
	\end{aligned}
	\end{equation}
\end{itemize}
After having lifted all the variables, we obtain the \highlight{sequential LCI:}
\begin{equation}
	\label{lci}
	\sum_{i\in C} x_i + \sum_{i \in N_0} \alpha_i x_i + \sum_{i\in N_1} \beta_i x_i  \leq c -1 + \sum_{i \in N_1}\beta_i.
\end{equation}
\highlight{The procedure to obtain inequality of \eqref{lci} is first described implicitly in \cite{Wolsey1975}.} 
See for example \highlight{\cite{Nemhauser1988,Gu1995,Kaparis2011}} for more details about \highlight{sequential} LCI. Here we just
notice that different lifting sequences may lead to different \highlight{sequential} LCIs.

The inequality \eqref{lci} is called a non-project LCI if $ N_1 = \varnothing $ and a project LCI if $ N_1 \neq \varnothing $ (see \cite{Gu1995}). Given a lifting sequence, the non-project LCI can be computed (\cite{Zemel1989}) in the complexity of $ \mathcal{O}(c n) $, where $c=|C|$ again.
For the project LCI, if $ C \cup N_1 $ is a minimal cover of $ X $ and the lifting sequence is enforced to $\{l^0_1,\ldots,l^0_r, l^1_1,\ldots,l^1_{|N_1|},l^0_{r+1},\ldots,l^0_{|N_0|}\} $, where $ l^j_i \in N_j  $ for $ i =1,\ldots,|N_j|$ and $j=0,1$ and
$r$ is a given integer between $1$ and $|N_0|$, Gu \cite{Gu1995} proved that 
\highlight{\eqref{lci}} can be obtained in the complexity of $ \mathcal{O}(cn^3) $.
As mentioned before, however, the complexity of computing a \highlight{sequential} LCI with an arbitrary lifting sequence is still unknown.

We close this section by noting that the \highlight{sequential} LCI is invariant under scaling.
\begin{myobservation}
	\label{invariantproperty}
	Given the same partition and the lifting sequence, multiplying a positive integer to the knapsack constraint gives the same \highlight{sequential} LCI.
\end{myobservation}
\section{The example by Gu (1995)}
\label{sec-example}
In this section, we describe the elegant example constructed by Gu \cite{Gu1995},
which leads the lifting coefficients to be exponential. It is related to the
following vector $f\in\mathbb{Z}^{2r+1}$, where $r$ is a given positive integer.
\begin{equation}
	\label{exampleitems}
	f_1 = 1, \, f_2 =1 , \,f_3=1, \, \text{and} \, f_i = f_{i-2}+ f_{i-1} ,\ \text{for} \ i = 4,\ldots,2r + 1.
\end{equation}
\highlight{Notice that $ f $ is analogous to the Fibonacci sequence where the only difference 
is that the first element in the Fibonacci sequence is $ 0 $.}
We give two facts on the vector $f$, which can easily be verified by induction.
\begin{myobservation}
	\label{inductionformula}
	For $ j=3,\dots,2r+1 $, it holds that $ f_{j} = \sum_{i=1}^{j-2} f_{i} $	.
\end{myobservation}
\begin{myobservation}
	\label{boundsf}
	For $ j=3,\dots,2r+1 $, it holds that  $ \hilight{\frac{1}{4} (\sqrt{2} -1) (\sqrt{2})^{j} \leq f_j \leq 2^j} $.
\end{myobservation}

Consider the knapsack set $ X $ with $ 2r+1 $ variables, where the coefficients of the knapsack constraint
are $ f_1,\ldots,f_{2r+1}$ in \eqref{exampleitems} and the associated knapsack capacity is $ b = \sum_{i=1}^{2r}f_i$.
Consider the partition $ (C,N_0, N_1) $ of $ \{1,\ldots,2r+1 \} $ with $ C=\{1,2\} $, $ N_0=\{ 4, 6,8,\ldots 2r \} $, and $ N_1 = \{ 3,5,7,\ldots, 2r+1 \} $. Since $f_1 =1 $, $ f_2=1 $, and
\begin{eqnarray*}
	\bar{b}&=&b-\sum_{i\in N_1}f_i  =b-\sum_{i=1}^{r}f_{2i+1} =b-f_3 - \sum_{i=2}^r (f_{2i-1} + f_{2i})\\
	&=& b-f_3 - \sum_{i=3}^{2r}f_i = 1,
\end{eqnarray*}
we know that $ C $ is a minimal cover of $ X(N_0,N_1)= \{ x\in \mathbb{B}^2 \, : \, x_1 + x_2 \leq 1\} $. Now consider the lifting sequence $ \{3,4,\ldots,2r+1\}$; i.e., $ l_j = j+2 $ for $ 1 \leq j \leq 2r-1 $. The following lemma is due to \cite{Gu1995}. For completeness, we provide a proof.
\begin{lemma}[Gu 1995]
	\label{guexample}
	(i) $ \alpha_{i} = f_{i} $ for $ i \in N_0 $; (ii) $ \beta_{i} = f_{i} $ for $ i \in N_1 $.
\end{lemma}
\begin{proof}
	We proceed by induction. At first, the lifting problem of the variable $ x_3 $ reads
	\begin{equation*}
		\begin{aligned}
			\hilight{\beta_{3}} =  & \ \hilight{\text{max}}
			& &   x_1 + x_2 -1  \\
			&  \ \text{s.t.}  & &         x_1 + x_2 \leq 2, \ x \in \mathbb{B}^2   .
		\end{aligned}
	\end{equation*}
	Hence $ \beta_{3}= 1 = f_3 $.
	Assume that $ \alpha_i = f_i $ for $ i \leq j $, $ i \in N_0 $ and $ \beta_i = f_i  $ for $ i \leq j $, $ i\in N_1 $, respectively. Now we consider lifting the variable $ x_{j+1} $. If $ {j + 1} \in N_0 $, the problem \eqref{uplift} reduces to
	\begin{equation}
		\label{exauplift}
		\begin{aligned}
			\alpha_{j+1} =  & \ {\text{min}}
			& &  1 + \sum_{i=1}^{(j-1)/2} f_{2i+1} - \sum_{i=1}^{j}f_i x_i  \\
			&  \ \text{s.t.}  & &          \sum_{i=1}^j f_i x_i  \leq 1 + \sum_{i=1}^{(j-1)/2} f_{2i+1} - f_{j+1},\ x\in \mathbb{B}^{j}  .
		\end{aligned}
	\end{equation}
	Since
	\begin{equation*}
		\sum_{i=1}^{(j-1)/2}f_{2i+1} = f_3 + \sum_{i=2}^{(j-1)/2}(f_{2i} + f_{2i-1})= f_ 3 + \sum_{i=3}^{j-1} f_i  = \sum_{i=1}^{j-1}f_i -1 = f_{j+1} -1,
	\end{equation*}
	where the last equality follows from Observation \ref{inductionformula}, the feasibility of the problem \eqref{exauplift} requires
$ x_i = 0 $, $ 1\leq i \leq j $ and hence $\alpha_{j+1} = f_{j+1} $. If $ j+1 \in N_1 $, the problem \eqref{downlift} reduces to
	\begin{equation}
		\label{exadownlift}
		\begin{aligned}
			\beta_{j+1} =    &\  {\text{max}}
			& & \sum_{i=1}^{j}f_i x_i - 1 - \sum_{i=1}^{j/2 -1} f_{2i+1}    \\
			&  \ \text{s.t.}          & & \sum_{i=1}^j f_i x_i  \leq 1 + \sum_{i=1}^{j/2-1} f_{2i+1} + f_{j+1}, \ x\in \mathbb{B}^{j}.
		\end{aligned}
	\end{equation}
 It is easy to verify that $\sum_{i=1}^j f_i  = 1 + \sum_{i=1}^{j/2-1}f_{2i+1} + f_{j+1}$. Hence the all-one vector is feasible
 and solves the problem \eqref{exadownlift}, yielding $\beta_{j+1}=f_{j+1}$. So the statement is true for $j+1$ as well. By induction, this lemma holds.
	\qed
\end{proof}

\vspace{0.2cm}
Lemma \ref{guexample} indicates that the \highlight{sequential} LCI for this specific example is
\begin{equation}
	\label{liftedinequality}
	\sum_{i=1}^{2r+1} f_i x_i \leq \sum_{i=1}^{2r}f_i.
\end{equation}
By Observation \ref{boundsf}, \hilight{the input size of this example is polynomial, but} the lifting coefficients $\{f_i\}$ are exponential with $r$.
This example by Gu will play an important role in the coming complexity analysis.
\section{$ \mathcal{NP} $-hardness of computing a \highlight{sequentially} lifted cover inequality}
\label{sec-np}
In this section, we show the $ \mathcal{NP} $-hardness of the problem of computing \highlight{a sequential} LCI.
To begin with, we give a basic property of the vector $ f $ in \eqref{exampleitems}.
\begin{lemma}
	\label{anyinteger}
	Let $ f  $ be defined as in \eqref{exampleitems}, where $r\ge 1$ is given.
    For any $ \tau \in \mathbb{Z}_+$ satisfying $ 0 \leq \tau \leq \sum_{i=1}^{2r+1} f_{i}$, there
	exists a subset $ S \subseteq \{1,\ldots,2r+1\} $ such that $ \tau = \sum_{i\in S} f_i  $.
\end{lemma}
\begin{proof}
	We proceed by induction on $r$. The result apparently holds for $r=1$.
	Assume that the result is true for some $r\geq 1 $. To verify the result for $r+1$,
    it suffices to consider the case that
    $  \sum_{i=1}^{2r+1} f_i  <  \tau \leq \sum_{i=1}^{2(r+1)+1} f_i$.
	In fact, from Observation \ref{inductionformula}, we have that
	\begin{eqnarray*}
		\sum_{i=1}^{2r+1} f_i  = f_{2r+3},\qquad&&\sum_{i=1}^{2r+3} f_i = f_{2r+3} + f_{2r+2}+ f_{2r+3}.
	\end{eqnarray*}
	So $ f_{2r+3} < \tau \leq f_{2r+3} + f_{2r+2}+ f_{2r+3}$. Let us define $ \tau_1 $ as
	\begin{equation}
	\label{tau1}
	\tau_1 =\left\{\begin{array}{ll}
	\tau- f_{2r+2}- f_{2r+3},  & \ \text{if} \ \tau > f_{2r+2}+ f_{2r+3}; \\
	\tau - f_{2r+3}, &\ \text{if} \  \tau \leq f_{2r+2}+ f_{2r+3}.
	\end{array}\right.
	\end{equation}
	Then it is easy to see that $ \tau_1 \leq f_{2r+3} = \sum_{i=1}^{2r+1} f_i $.
	By the induction assumption, there exists $S \subseteq \{ 1,2,3,...,2r+1 \} $ such that $ \tau_1 = \sum_{i \in S} f_i $. By picking
    \highlight{one more element $ f_{2r+3} $ and the possible element $ f_{2r+2} $}, we know that there exists some subset of $\{ 1,2,3,...,2r+3\}$ such that the sum of $f_i$'s over this subset is exactly $\tau$. Thus the result holds for $r+1$. By induction, this lemma is true.
	\qed
\end{proof}
\vspace{0.2cm}

%
\noindent 
\highlight{Next, we introduce the \highlight{restricted partition problem (RPP)}, which is a variant of the partition problem \cite{GAREY1979}. Comparing to the partition problem, 
	the RPP problem restricts the total sum of all element to some specific values and 
	allows the sum of the elements in the subset equals to one more value.
The RPP problem is shown to be $ \mathcal{NP} $-hard in 
	\cite{Klabjan1998a}.}
\begin{quote}
{\bf Problem RPP}.\ Given a nonnegative integer $ m $, a finite set $ K $ \highlight{of $ k $ elements with value $ \omega_i \in \mathbb{Z}_+ $} for the $ i $-th element \highlight{and}
	$ \sum_{i\in K} \omega_i  = 2(2^{m+1} -   1) $, does there exist a subset $ T \subseteq K $ such that $ \sum_{i \in T} \omega_i = 2^{m+1} - 1 $ or $ \sum_{i \in T} \omega_i = 2^{m+1} - 2 $ ?
\end{quote}
\highlight{For convenience, define $\lambda = 2^{m+1} -1 $ and then $ \sum_{i\in K} \omega_i = 2\lambda  $. 
The RPP problem is still $ \mathcal{NP} $-hard when 
$ 1 \leq \omega_i \leq \lambda - 1 $. To see this, 
suppose there exists some $ j \in K $ such that $ \omega_j \geq \lambda $. We have the following two cases. 
\begin{itemize}
	\item [(i)]
		$ \lambda \leq \omega_j \leq \lambda+1 $. In this case, it follows that 
		$ \sum_{i \in K \backslash\{j\}} \omega_i = \lambda-1 $ or $ \sum_{i \in K \backslash\{j\}} \omega_i = \lambda  $. Thus, $ K \backslash\{j\}$ is the desired subset and the answer to the RPP problem is yes. 
	\item [(ii)]
	$ \omega_j \geq \lambda+2 $. In this case, it is easy to see that the answer to the RPP 
	problem is no. 
\end{itemize}
As we can see, both cases can be solved in polynomial time. 
Therefore, the statement follows and in the following, we assume that 
 $ 1 \leq \omega_i \leq \lambda - 1 $.
}
We are now ready to present the main result of this paper; i.e., provide an $ \mathcal{NP} $-hardness proof for computing \highlight{a sequential} LCI. The basic idea is as follows. Firstly, we adopt Gu's example (see Section \ref{sec-example}) to make the lifting coefficients exponential. Secondly, some variables fixed at zero will be lifted, where the lifting coefficients can easily be obtained. Finally, we lift a variable fixed at one, \highlight{where the objective value of corresponding lifting problem is equal to some specific value if and only if the answer to the RPP problem is yes and hence is $ \mathcal{NP} $-hard}.
\begin{theorem}
	The problem of computing \highlight{a sequential} LCI is $ \mathcal{NP} $-hard.
\end{theorem}
\begin{proof}
    For an RPP instance with $m\ge 0$ and $K= \{ 2r+4,\ldots, n-1 \} $, we shall construct a problem of computing a \highlight{sequential} LCI
    in polynomial time. 
	We construct a problem of computing a \highlight{sequential} LCI as follows.
Set $ r = m + 6 $,
	$ n=2r+4 +k $, and $ b = \sum_{i=1}^{2r+1} \lambda f_i  + \lambda (3\lambda + 6 )   $, where $ f$ is defined as in \eqref{exampleitems}. Set the coefficients of the knapsack constraint as
	\begin{equation*}
	a_i =\left\{\begin{array}{ll}
	\lambda f_i ,  & \ \text{for} \ i=1,\ldots,2r+1; \\
	\lambda(\lambda+3) + 1 , &\ \text{for} \  i=2r+2; \\
	\lambda(\lambda+3) - 1  , &\ \text{for} \  i=2r+3; \\
	\omega_i(\lambda + 1)  , &\  \text{for} \  i=2r+4,\ldots,n-1; \\
\lambda ( 3\lambda + 6 + f_{2r+1}) , &\ \text{for} \  i=n.
	\end{array}\right.
	\end{equation*}
 Define the partition $ (C, N_0, N_1) $ with $ C = \{1,2\} $, $ N_0 = \{ 4,6,8,\ldots, 2r,2r+2, 2r+3,\ldots,n-1  \} $, and $ N_1 = \{3,5,7,\ldots, 2r+1, n\} $. Finally, let the lifting sequence be $ \{\hilight{3,\ldots,n}\}  $.
We shall prove that the lifting coefficient $ \hilight{\beta_n} = f_{2r+1} + 3\lambda + 5  $ if and only if the answer to the RPP instance is yes.

Before doing this, we note that the input size of this \highlight{instance} is polynomial of that of the RPP instance. To see this, let $ L $ be the input size of the RPP instance. \highlight{It follows immediately that $ k = \mathcal{O}(L) $.
We next show that the number of elements $ n $ is satisfied with $ n = \mathcal{O}(L) $.
	Since the input size of a positive integer $ t $ is $  \log_2 (t+1) $, it follows that
	\begin{equation*}
		\log_2(2^{m+2} -2+1) =  \log_2 (\sum_{i \in K} \omega_i +1 ) \leq \sum_{i\in K}\log_2(\omega_i + 1) \leq L,
	\end{equation*}
	where the first inequality follows from $\omega_i +1 \geq 2  $ for all $ i \in K $.
	Thus, we have that $ m \leq L -2 $, which further implies that $ r=m+6 = \mathcal{O}(L) $. 
	This, combined with the fact that $ k = \mathcal{O}(L) $, indicates
	$$
		n = 2r+ 1 + k =\mathcal{O}(L).
	$$
	Finally, it can be easily verified that $ a_i= \mathcal{O}(L^2) $ for $ i =1,\ldots, n $
	and $ b = \mathcal{O}(L^2)  $. This proves that the input size of the 
	constructed instance is polynomial of that of the RPP instance.}
	For preparation \highlight{of the proof}, we also verify
\begin{equation}
\begin{aligned}
f_{2r+1} & \geq & & \dfrac{1}{4} (\sqrt{2} -1) (\sqrt{2})^{2r+1}=  \dfrac{1}{4} (2-\sqrt{2}) 2^r= \dfrac{1}{4} (2-\sqrt{2}) 2^{m+6} &  \\[8pt]
& > &&  2^{m+2} + 4 = 2\lambda + 6. &  \label{lambdaf}
\end{aligned}
\end{equation}

	In the following, we consider the lifting procedure. By construction, the knapsack inequality of this instance reads
	\begin{eqnarray*}
		&&\sum_{i=1}^{2r+1} \lambda f_i x_i  + [\lambda(\lambda+3)+1] x_{2r+2} + [\lambda(\lambda+3)-1] x_{2r+3} +\sum_{i=2r+4}^{n-1} \omega_i (\lambda+1) x_i     \\
		&&\qquad + \lambda ( 3\lambda + 6 + f_{2r+1}) x_n \leq \sum_{i=1}^{2r+1} \lambda f_i  + \lambda (3\lambda + 6 ).
	\end{eqnarray*}
	Since $ a_1 = \lambda  $, $ a_2 = \lambda $, and
	\begin{equation*}
		\begin{aligned}
			b- \sum_{i\in N_1}a_i = \sum_{i=1}^{2r+1} \lambda f_i  + \lambda (3\lambda + 6 )- \sum_{i=1}^{r} \lambda f_{2i+1} -\lambda ( 3\lambda + 6 + f_{2r+1}) &   \\
			= \sum_{i=1}^{2r} \lambda f_i - \sum_{i=1}^{r}\lambda f_{2i+1} =  \sum_{i=1}^{2r} \lambda f_i  -\lambda f_3 - \sum_{i=2}^{r} \lambda (f_{2i-1}+f_{2i})= \lambda,
		\end{aligned}
	\end{equation*}
	we know that $ C $ is a minimal cover of $ X(N_0,N_1) $ with the cover inequality
	\begin{equation}
		\label{proofcover}
		x_{1}+x_2 \leq 1.
	\end{equation}	
	
\vspace*{0.3cm}
\noindent {\bf Step 1. Lifting the variables $ \boldsymbol{x_3,\ldots,x_{2r+1}} $}

\vspace*{0.2cm}\noindent Starting with the cover inequality \eqref{proofcover}, we know from Observation \ref{invariantproperty} and the inequality \eqref{liftedinequality} that, after lifting the variables $ x_3,\ldots, x_{2r+1} $, the inequality is
	\begin{equation*}
		\sum_{i=1}^{2r+1} f_i x_i \leq \sum_{i=1}^{2r} f_i.
	\end{equation*}

\vspace*{0.3cm}
\noindent {\bf Step 2. Lifting the variables $ \boldsymbol{x_{2r+2}} $ and $ \boldsymbol{x_{2r+3}} $}

\vspace*{0.2cm}\noindent We first consider the variable $ x_{2r+2} $. The associated lifting problem is

	\begin{equation}
		\label{lifting2r2}
		\begin{aligned}
			\alpha_{2r+2} =  \ & {\text{min}} &   & \sum_{i=1}^{2r} f_i  - \sum_{i=1}^{2r+1} f_i x_i                                                 \\
			                & \text{s.t.}  &   & \sum_{i=1}^{2r+1} \lambda f_i x_i  \leq \lambda \sum_{i=1}^{2r}  f_i -[\lambda(\lambda+3) +1] ,\ x \in \mathbb{B}^{2r+1}.
		\end{aligned}
	\end{equation}
	The problem \eqref{lifting2r2} is feasible at the zero vector since
	\begin{equation*}
\lambda \sum_{i=1}^{2r}  f_i -[\lambda(\lambda+3)+1] \geq \lambda (\sum_{i=1}^{2r}  f_i - \lambda- 4)  > \lambda (f_{2r+1}- \lambda - 4)> 0,
	\end{equation*}
where the last inequality follows from \eqref{lambdaf}. Let $ \bar{x} $ be the optimal solution of \eqref{lifting2r2}. Since $ f_{i} \in \mathbb{Z} $ for $ i =1,\ldots,2r+1 $, its feasibility requires
	\begin{equation*}
		\sum_{i=1}^{2r+1} f_i\bar{x}_i \leq \lfloor\frac{ \lambda \sum_{i=1}^{2r} f_i - \lambda(\lambda+3) - 1}{\lambda} \rfloor =  \sum_{i=1}^{2r} f_i  - \lambda - 4.
	\end{equation*}	
	On the other hand,  from Lemma \ref{anyinteger}, we can always find an $ \bar{x} $ such that
	\begin{equation*}
	 	\sum_{i=1}^{2r+1} f_i \bar{x}_i = \sum_{i=1}^{2r}  f_i - \lambda - 4.
	\end{equation*}
	The optimality of $ \bar{x} $ gives that $ \alpha_{2r+2} = \sum_{i=1}^{2r}  f_i - \sum_{i=1}^{2r+1} f_i \bar{x}_i = \lambda + 4 $.
	Similarly, the lifting problem of $ x_{2r+3} $ reads
		\begin{align}
			\alpha_{2r+3} = \ & {\text{min}} &   & \sum_{i=1}^{2r} f_i  - \sum_{i=1}^{2r+1} f_i x_i - (\lambda + 4)x_{2r+2}                                                    \nonumber        \\
			                & \text{s.t.}  &   & \sum_{i=1}^{2r+1} \lambda f_i x_i + [\lambda(\lambda+3) +1] x_{2r+2}  \leq \lambda \sum_{i=1}^{2r}  f_i - [\lambda(\lambda+3) -1] , \nonumber\\
			                &              &   & x \in \mathbb{B}^{2r+2}. \label{lifting2r3}
		\end{align}
	Then $ \alpha_{2r+3} = \lambda + 2 $, which is achieved at an optimal solution $ \bar{x} $ satisfying
	$ \bar{x}_{2r+2}=1 $ and $ \sum_{i=1}^{2r+1}f_i \bar{x}_i = \sum_{i=1}^{2r} f_i  - 2\lambda-6 $.
	
\vspace*{0.3cm} \noindent {\bf Step 3. Lifting the variables $ \boldsymbol{x_{2r+4},\ldots,x_{n-1}} $}

\vspace*{0.2cm}\noindent	We shall show that $ \alpha_{i} = \omega_i $ for all $ i\in K$ by induction.
	At first, consider the lifting \hilight{of} the variable $ x_{2r+4} $. This requires to solve the problem
	\begin{equation}
		\label{lifting2r4}
		\begin{aligned}
			\alpha_{2r+4}= \ & {\text{min}} &   & \sum_{i=1}^{2r} f_i  - \sum_{i=1}^{2r+1} f_i x_i - (\lambda + 4)x_{2r+2}  - (\lambda + 2)x_{2r+3}        \\
			               & \text{s.t.}  &   & \sum_{i=1}^{2r+1} \lambda f_i x_i + [\lambda(\lambda+3) +1] x_{2r+2} +  [\lambda(\lambda+3)-1] x_{2r+3} \\
			               &              &   & \leq \lambda \sum_{i=1}^{2r}  f_i - \omega_{2r+4} (\lambda + 1) , \ x \in \mathbb{B}^{2r+3}.
		\end{aligned}
	\end{equation}
	If $ \hat{x} $  is an optimal solution of the problem \eqref{lifting2r4} with $ \hat{x}_{2r+3}=1 $,  by the feasibility and \eqref{lambdaf}, we have
	\begin{equation*}
		\sum_{i=1}^{2r+1} f_{i}\hat{x}_i  + \lambda + 2 <  \sum_{i=1}^{2r} f_{i} + \lambda + 2 <\sum_{i=1}^{2r+1} f_{i}.
	\end{equation*}
	This, together with Lemma \ref{anyinteger}, indicates that we can define a new feasible point $ \bar{x} $ such that $ \bar{x}_{2r+2}= \hat{x}_{2r+2} $, $  \bar{x}_{2r+3}= 0$, and $ \sum_{i=1}^{2r+1} f_i \bar{x}_i =\sum_{i=1}^{2r+1} f_i \hat{x}_i + \lambda + 2$.
	It is easy to see that $ \bar{x} $ and $ \hat{x}$ give the same objective values. Hence, we can assume that $ x_{2r+3} = 0 $ in the problem \eqref{lifting2r4}. Furthermore, since $ \omega_{2r+4} \leq  \lambda -1 $, similar to the problem \eqref{lifting2r3},
we can show that the optimal value of \eqref{lifting2r4} is $ \alpha_{2r+4} = \omega_{2r+4}$, which is achieved at an optimal solution
 $ \bar{x} $ satisfying $\bar{x}_{2r+3}=0$, $ \bar{x}_{2r+2}=1$, and $ \sum_{i=1}^{2r+1}f_i \bar{x}_i = \sum_{i=1}^{2r} f_i  - (\lambda+4)-\omega_{2r+4}$.

 Now assume that $ \alpha_{i} = \omega_{i} $ for $ 2r+4 \leq i \leq j $ and $j<n-1$ and consider the lifting problem of $ x_{j+1} $:
\begin{spacing}{0.9}
	\begin{align}
	\alpha_{j+1}= \ & {\text{min}} &   & \sum_{i=1}^{2r} f_i  - \sum_{i=1}^{2r+1} f_i x_i - (\lambda + 4)x_{2r+2}  - (\lambda + 2)x_{2r+3} - \sum_{i=2r+4}^{j} \omega_i x_i  \nonumber      \\
	& \text{s.t.}  &   & \sum_{i=1}^{2r+1} \lambda f_i x_i + [\lambda(\lambda+3) +1] x_{2r+2} +  [\lambda(\lambda+3) -1] x_{2r+3} +   \nonumber\\
	&              &   & \sum_{i=2r+4}^{j}  \omega_i(\lambda+1)x_i  \leq \lambda \sum_{i=1}^{2r}  f_i - \omega_{j+1} (\lambda + 1) , \ x \in \mathbb{B}^{j}. \label{lifting2rj1}
	\end{align}
\end{spacing}
\noindent We claim that there exists an optimal solution $ \bar{x} $ such that $ \bar{x}_{2r+3} = \bar{x}_{2r+4}=\cdots= \bar{x}_j = 0 $. To see this, suppose that an optimal solution $\hat{x}$ is such that $ \hat{x}_t = 1 $ for some $ t \subseteq [2r+3,\,j]$. Analogously, define a new point $ \bar{x} $ with $ \hilight{\bar{x}_{i}= \hat{x}_{i}} $ for $ 2r+2 \leq i \leq j$ and $i\ne t$, $ \bar{x}_{t} = 0 $, and $ \sum_{i=1}^{2r+1} f_i \bar{x}_i =\sum_{i=1}^{2r+1} f_i \hat{x}_i + \theta_t$, where
	\begin{equation*}
	\theta_t=\left\{\begin{array}{ll}
	\lambda + 2,  &\ \text{if} \ t = 2r+3; \\
	\omega_t, &\ \text{otherwise}.
	\end{array}\right.
	\end{equation*}
	By simple calculations, $ \hilight{\bar{x}} $ is feasible to the problem \eqref{lifting2rj1} and gives the same objective value as $ \bar{x} $. Similar to the problem \hilight{\eqref{lifting2r3}}, we can verify that $ \alpha_{j+1} = \omega_{j+1}$.
Thus by induction, we have that $ \alpha_{i} = \omega_i $ for all $ i\in K$.
	
\vspace*{0.3cm} \noindent {\bf Step 4. Lifting the variable $ \boldsymbol{x_{n}} $}

\vspace*{0.2cm}\noindent	Finally, we concentrate on lifting the variable $ x_n $.
	The lifting problem is
		{
		\begin{spacing}{0.9}
		\begin{align}
			\beta_{n} = \  & {\text{max}} &   & \sum_{i=1}^{2r+1} f_i x_i + (\lambda + 4)x_{2r+2} + (\lambda + 2)x_{2r+3}  + \sum_{i=2r+4}^{n-1}  \omega_i x_i  -\sum_{i=1}^{2r} f_i\nonumber \\
			             & \text{s.t.}  &   & \sum_{i=1}^{2r+1} \lambda f_i x_i + [\lambda(\lambda+3)+1] x_{2r+2}+  [\lambda(\lambda+3)-1] x_{2r+3} +                        \nonumber    \\
			             &              &   & \sum_{i=2r+4}^{n-1}\omega_i(\lambda +1)x_i  \leq \lambda \sum_{i=1}^{2r+1}  f_i + \lambda(3\lambda +6) , \ x \in \mathbb{B}^{n-1}.	\label{liftingn}
		\end{align}
	\end{spacing}
}
\noindent For convenience, denote $ g(x) $ to be the objective function in the above problem. Consider the point $ \check{x} $ with $ \check{x}_i =1 $ for $ 2r+2\leq  i\leq n-1  $ and $ \sum_{i=1}^{2r+1} f_i \check{x}_i = \sum_{i=1}^{2r+1} f_i  - \lambda - 2$. 
By Lemma \ref{anyinteger}, such a point must exist. We can check that $ \check{x} $ is feasible to the problem \eqref{liftingn} and
$ g(\check{x}) = f_{2r+1} + 3\lambda+4$. Furthermore, for a binary vector $x$, if at least one of the two components $x_{2r+2}$ and $x_{2r+3}$ is equal to zero, we have that
$$ g(x) \leq  \sum_{i=1}^{2r+1} f_i  + \lambda + 4 + \sum_{2r+4}^{n-1} \omega_i - \sum_{i=1}^{2r}f_i  = f_{2r+1} + 3\lambda + 4. $$
Thus to seek better values for $\beta_n$, we may set $ x_{2r+2} = x_{2r+3} = 1 $. In this case, the problem \eqref{liftingn} reduces to
	%
		\begin{align}
			\beta_{n} = \ & {\text{max}} &   & \sum_{i=1}^{2r+1} f_i x_i + \sum_{i=2r+4}^{n-1}  \omega_i x_i + 2\lambda + 6 -    \sum_{i=1}^{2r} f_i    	\label{liftingn-1}           \\
			             & \text{s.t.}  &   & \sum_{i=1}^{2r+1} \lambda f_i x_i  + \sum_{i=2r+4}^{n-1}\omega_i(\lambda +1)x_i  \leq \lambda \sum_{i=1}^{2r+1}  f_i + \lambda^2,\ x \in \mathbb{B}^{n-3}. \nonumber
		\end{align}

Now assume that $\bar{x}$ is an optimal solution of \eqref{liftingn-1}. Denote $ p = \sum_{i=2r+4}^{n-1} \omega_i \bar{x}_i $. It is easy to see
that $ p \leq 2\lambda $. Consider the following four cases.
	\begin{enumerate}
		\item [(a)] $ p  \le \lambda - 2 $. In this case, the knapsack constraint in the problem \eqref{liftingn-1} is trivially satisfied and the optimality of $ \bar{x} $ implies 
		\highlight{$ \bar{x}_i = 1$ for $ i = 1,\ldots, 2r+1$}. \highlight{This further indicates} that
		$$ g(\bar{x}) = \sum_{i=1}^{2r+1} f_i + p + 2\lambda + 6 -  \sum_{i=1}^{2r} f_i = f_{2r+1} + p +  2\lambda + 6 \le f_{2r+1} + 3\lambda + 4.
		$$
		\item  [(b)]
		      $ p  = \lambda - 1 $.  Similar to the case (a), we have that 	$ g(\bar{x}) =  f_{2r+1}+3\lambda + 5  $.
		\item  [(c)]
		      $ p = \lambda $. In this case, the feasibility of $\bar{x}$ indicates that
		      $$
		      \sum_{i=1}^{2r+1} f_i \bar{x}_i  \leq  [  \lambda \sum_{i=1}^{2r+1} f_i + \lambda^2  - \lambda(\lambda+1) ]/\lambda = \sum_{i=1}^{2r+1} f_i -1.
		      $$ Furthermore, the optimality of $\bar{x}$ implies that $ \sum_{i=1}^{2r+1} f_i \bar{x}_i = \sum_{i=1}^{2r+1} f_i -1  $. Thus we
can also check that $ g(\bar{x})=f_{2r+1}+3\lambda + 5 $.
		\item  [(d)]
		      $ \lambda+1 \le p\leq 2\lambda  $. On one hand, the feasibility of $\bar{x}$ requires
		      \begin{equation}
		      	\label{ineqa}
		      	\sum_{i=1}^{2r+1} f_i \highlight{\bar{x}_i}   \leq \lfloor \frac{\lambda \sum_{i=1}^{2r+1} f_i + \lambda^2  - p(\lambda +1)}{\lambda}\rfloor = \sum_{i=1}^{2r+1} f_i  + \lambda - p - 2,
		      \end{equation}
		      where the last equality follows from \highlight{$ \lambda + 1 \leq p \leq 2 \lambda$}.
		      On the other hand, the optimality of $\bar{x}$ requires that the inequality in \eqref{ineqa} holds with equality, yielding
		      $$ g(\bar{x}) = \sum_{i=1}^{2r+1} f_i + \lambda - p - 2+ p + 2\lambda + 6 -  \sum_{i=1}^{2r} f_i =f_{2r+1} + 3\lambda + 4.
		      $$
	\end{enumerate}

To summarize, the lifting coefficient $ \beta_n = f_{2r+1} + 3\lambda + 5  $ for 
the problem of computing the \highlight{sequential} LCI constructed in the above,  if any only if $ p = \lambda  $ or $ p = \lambda -1 $, or equivalently,
the answer to the RPP instance is yes.
Since the RPP problem is $\mathcal{NP}$-hard and the construction of the companion problem is in polynomial, 
we conclude that computing \highlight{a sequential} LCI is $\mathcal{NP}$-hard. This completes the proof.
	%
	\qed
\end{proof}
\def\urlprefix{}\def\href#1#2#3#4{\ifstrequal{#2}{[link]}{}{#2\newline}}
\bibliographystyle{spmpsci}
\bibliography{ComplexityprojectiLCI.bib}

\end{document}